\documentclass[11pt]{amsart}

\usepackage{amsmath,amsthm,amssymb,graphicx,listings,color,enumerate,url}
\DeclareFontShape{OT1}{cmtt}{bx}{n}{<5><6><7><8><9><10><10.95><12><14.4><17.28><20.74><24.88>cmttb10}{}
\usepackage{todonotes}

\newtheorem{thm}{Theorem}

\newtheorem{lem}[thm]{Lemma}

\def\ph{\varphi}

\let\on=\operatorname

\def\R{{\mathbb R}}

\def\Imm{{\on{Imm}}}

\def\Diff{{\on{Diff}}}

\def\Hor{{\on{Hor}}}
\def\Nor{{\on{Nor}}}
\def\Ver{{\on{Ver}}}
\def\Tan{{\on{Tan}}}

\begin{document}
\title{Metrics with prescribed horizontal bundle on  spaces of curves}
\author{Martin Bauer}
\author{Philipp Harms}
\address{
Martin Bauer
Fakult\"at f\"ur Mathematik und Geoinformation, TU Wien
\newline\indent
Philipp Harms: Department of Mathematics, ETH Z\"urich}
\email{bauer.martin@univie.ac.at}
\email{philipp.harms@math.ethz.ch}
\thanks{M. Bauer was supported by the European Research Council (ERC), within the project 306445 (Isoperimetric Inequalities and Integral Geometry) and by the FWF-project P24625 (Geometry of Shape spaces).
}
\date{\today}

\keywords{curve matching; shape space; Sobolev type metric; reparametrization group; Riemannian shape analysis; Riemannian submersion; infinite dimensional geometry}
\subjclass[2010]{Primary 58D17, 58E30, 35A01}

\begin{abstract}
We study metrics on the shape space of curves that induce a prescribed splitting of the tangent bundle. More specifically, we consider reparametrization invariant metrics $G$ on the space $\Imm(S^1,\R^2)$ of parametrized regular curves. For many metrics the tangent space $T_c\Imm(S^1,\R^2)$ at each curve $c$ splits into vertical and horizontal components (with respect to the projection onto the shape space $B_i(S^1,\R^2)=\Imm(S^1,\R^2)/\Diff(S^1)$ of unparametrized curves and with respect to the metric $G$). In a previous article we characterized all metrics $G$ such that the induced splitting coincides with the natural splitting into normal and tangential parts. In these notes we extend this analysis to characterize all metrics that induce any prescribed splitting of the tangent bundle.
\end{abstract}

\maketitle

\section{Introduction}
Let $\Imm(S^1,\mathbb R^2)$ be the space of regular planar curves. Our center of attention lies on the shape space of unparametrized curves, which can be identified with the quotient space
\begin{equation*}
B_i(S^1,\mathbb R^2)=\Imm(S^1,\mathbb R^2)/\Diff(S^1)\,.
\end{equation*}
Here, $\Diff(S^1)$ denotes the Lie group of all smooth diffeomorphisms on the circle, which acts smoothly on $\Imm(S^1,\mathbb R^2)$ via composition from the right:
\begin{equation*}
\Imm(S^1,\mathbb R^2)\times \Diff(S^1)\to \Imm(S^1,\mathbb R^2),\qquad (c,\ph)\mapsto c\circ\ph\,.
\end{equation*}
The quotient space $B_i(S^1,\mathbb R^2)$ is not a manifold, but only an orbifold with isolated singularities (see \cite{Michor1991} for more information). A strong motivation for considering this space -- and in particular Riemannian metrics thereon -- comes from the field of shape analysis \cite{Mio2004,Mio2007,SKKS2014,Sundaramoorthi2011,Michor2008a,Bauer2014b}. See \cite{Bauer2014} for a recent overview on various metrics on these spaces.

Given a reparametrization invariant metric $G$ on $\Imm(S^1,\mathbb R^2)$, we can (under certain conditions) induce a unique Riemannian metric on the quotient space $B_i(S^1,\mathbb R^2)$ such that the projection
\begin{equation}\label{equ:projection}
    \pi: \Imm(S^1,\mathbb R^2) \to B_i(S^1,\mathbb R^2):= \Imm(S^1,\mathbb R^2)/\Diff(S^1)
\end{equation}
is a Riemannian submersion. A detailed description of this construction is given in \cite[Section~4]{Bauer2011b}. For many metrics, $T\pi$ induces a splitting of the tangent bundle $T\Imm(S^1,\mathbb R^2)$ into a vertical bundle, which is defined as the kernel of $T\pi$, and a horizontal bundle, defined as the $G$-orthogonal complement of the vertical bundle:
\begin{equation}\label{equ:splitting1}
    T\Imm(S^1,\mathbb R^2) = \on{ker}T\pi \oplus (\on{ker}T\pi)^{\bot,G} 
    =: \Ver \oplus \Hor\, .
\end{equation}
If one can lift any curve in $B_i(S^1,\mathbb R^2)$ to a horizontal curve in $\Imm(S^1,\mathbb R^2)$, then there is a one-to-one correspondence between geodesics on shape space $B_i(S^1,\mathbb R^2)$ and horizontal geodesics on $\Imm(S^1,\mathbb R^2)$ \cite[Section~4.8]{Bauer2011b}. 

In \cite{BH2015} we described all metrics on $\Imm(S^1,\mathbb R^2)$ such that the splitting \eqref{equ:splitting1} coincides with the natural splitting into components that are tangential and normal to the immersed surface. In this article we characterize all metrics that induce an arbitrary given splitting of $T\Imm(S^1,\R^2)$, generalizing our previous result. 

As an application of our result we investigate a splitting that could be used to develop efficient numerics for the horizontal geodesic equation. The splitting is the decomposition of $T_c\Imm(S^1,\mathbb R^2)$ into deformations preserving the speed $\|\dot c\|$ and a suitable complement.

\section{The decomposition theorem}

\subsection{Assumptions}\label{subsec:assumptions}

Following \cite{Bauer2011b}, we now describe the class of metrics that we study in this article. We define all metrics via a so-called inertia operator $L$ by the formula
\begin{equation}\label{equ:metric}
G_c^L(h,k)  = \int_{S^1}\langle L_c h,k \rangle ds, 
\end{equation}
where $ds=|c'|d\theta$ denotes integration by arc-length.
We assume that $L$ is a smooth bundle automorphism of $T\Imm(S^1,\mathbb R^2)$ such that at every $c \in \Imm(S^1,\mathbb R^2)$, the operator 
\begin{equation*}
L_c:T_c\Imm(S^1,\mathbb R^2) \to T_c\Imm(S^1,\mathbb R^2)
\end{equation*}
is a pseudo-differential operator of order $2l$ which is symmetric and positive with respect to the $L^2$-metric on $\Imm(S^1,\mathbb R^2)$. Moreover, we assume that $L$ is invariant under the action of the reparametrization group $\Diff(S^1)$ acting on $\Imm(S^1,\mathbb R^2)$, i.e.,
\begin{equation*}
L_{c\circ \ph}(h\circ \ph)=L_c(h)\circ \ph \qquad \text{for all } \ph \in \Diff(S^1)\,.
\end{equation*}
These assumptions remain in place throughout this work. Their immediate use is as follows: being symmetric and positive, $L$ induces a Sobolev-type metric on the manifold of immersions through equation \eqref{equ:metric}. The $\Diff(S^1)$-invariance of $L$ implies the $\Diff(S^1)$-invariance of the metric $G^L$. Assuming that the decomposition in
horizontal and vertical bundles exists, there is a unique metric on $B_i(S^1,\mathbb R^2)$ such that the projection \eqref{equ:projection} is a Riemannian submersion (see \cite[Thm. 4.7]{Bauer2011b}). Then  the resulting geometry of shape space is mirrored by the ``horizontal geometry'' on the manifold of immersions. 

\subsection{Splitting into horizontal and vertical subbundles}
By  definition, the horizontal and vertical bundles are given by
\begin{equation*}\begin{split}
\Ver_c&:=\ker(T\pi)=\Tan_c\,,
\\ 
\Hor_c&:=(\Ver_c)^{\bot,G}=\big\{h\in T_c\Imm(S^1,\mathbb R^2): L_ch \in \Nor_c\}\,.
\end{split}\end{equation*}
Note, that $\Hor \oplus \Ver$ might not span all of $T\Imm$ in this infinite-dimensional setting. 

\subsection{Constructing metrics that induce a prescribed splitting.}
 
We now state our main result. 

\begin{thm}[Decomposition theorem]\label{thm:decomposition}
Let $c \in \Imm(S^1,\mathbb R^2)$ and let $H(c)$ be any complement of $\operatorname{Tan}(c)$, i.e.,
\begin{equation}\label{TImm_splitting_tangent}
T_c\Imm(S^1,\mathbb R^2)=\operatorname{Tan}(c) \oplus H(c)\;.
\end{equation}
Then the following conditions on a metric $G^L$ are equivalent:
\begin{itemize}
\item[(a)] The tangent bundle splits into a vertical and horizontal bundle \eqref{equ:splitting1} and this splitting coincides with  \eqref{TImm_splitting_tangent}.
 \item[(b)] 
The inertia operator $L_c$ admits a decomposition  
\begin{equation*}
 L_c=  (P^{\operatorname{tan}})^*\circ \tilde L_c \circ P^{\operatorname{tan}}+ (P^{H})^*\circ \tilde L_c \circ P^{H},
\end{equation*}
where $\tilde L_c: T_c\Imm(S^1,\R^2)\to T_c\Imm(S^1,\R^2)$ is an invertible pseudo-differential operator and where ${}^*$ denotes the adjoint with respect to the reparametrization-invariant $L^2$-metric.
\end{itemize}
\end{thm}

The theorem follows directly as a special case of the following lemma by setting $K(c)=\Tan(c)$. 

\begin{lem}
Let $c\in \Imm(S^1,\mathbb R^2)$ and let 
\begin{equation}\label{TImm_splitting_general}
T_c\Imm(S^1,\mathbb R^2)=H(c)\oplus K(c)
\end{equation}
be a given splitting with corresponding projections $P^H,P^K$. Then the following conditions on a metric $G^L$ are equivalent:
\begin{itemize}
\item[(a)] The subspace $H(c)$ is $G^L$-orthogonal to $K(c)$.
 \item[(b)] 
The operator $L_c$ has a decomposition  
\begin{equation*}
 L_c=  (P^{H})^*\circ \tilde L_c \circ P^{H}+ (P^{K})^*\circ \tilde L_c \circ P^{K},
\end{equation*}
where $\tilde L_c: T_c\Imm(S^1,\R^2)\to T_c\Imm(S^1,\R^2)$ is an invertible pseudo-differential operator and where ${}^*$ denotes the adjoint with respect to the reparametrization-invariant $L^2$-metric.
\end{itemize}
\end{lem}
\begin{proof}
Assume $(a)$ and let $h=P^{H}(h)+P^K(h):=h^{H}+h^K$.
We have
\begin{align*}
&G_c(h^{H}+h^K,k^{H}+k^K)\\&\qquad=G_c(h^{H},k^{H})+G_c(h^{H},k^K)+
G_c(h^K,k^{H})+G_c(h^K,k^K)\\&\qquad=G_c(h^{H},k^{H})+0+G_c(h^{K},k^{K})\;,
\end{align*}
where the last equality follows from the orthogonality of the splitting with respect to the metric $G$. Now the formula for the operator $L$ follows directly.

Conversely, assume $(b)$. To see the orthogonality we calculate
\begin{align*}
G_c(h^{H},k^K)&=\int_{S^1}\langle Lh^{H},k^K \rangle ds = 
\int_{S^1} \langle 0+ (P^{H})^*(\tilde L_c (h^{H})),k^{K}  \rangle ds\\&=
\int_{S^1} \langle (\tilde L_c (h^{H}), P^{H}(k^K)  \rangle  ds = 0\;.\qedhere
\end{align*}
\end{proof}

\section{Applications}
\subsection{The splitting into tangential and normal vector fields}
In this section, we want to recover the results from \cite{BH2015}. Letting $n$ denote the unit length normal vector field to the curve $c$, we define 
\begin{equation}
\Nor(c):=\left\{h=a.n: a\in C^{\infty}(S^1)\right\}\;.
\end{equation}
This yields a splitting
\begin{equation}\label{splitting_nor_tan}
T_c\Imm(S^1,\mathbb R^2)=\Tan(c)\oplus\Nor(c)\;
\end{equation}
with corresponding projections
\begin{align*}
&(P^{\operatorname{tan}})^*(h)=P^{\operatorname{tan}}(h)=\langle h,v\rangle v,\\
&(P^{\operatorname{nor}})^*(h)=P^{\operatorname{nor}}(h)=\langle h,n\rangle n\;.
\end{align*}
By Theorem~\ref{thm:decomposition} the splitting into horizontal and vertical bundles coincides with the above splitting \eqref{splitting_nor_tan} if and only if
$L$ can be written as
\begin{equation}
L_c=  (P^{\operatorname{tan}})^*\circ \tilde L_c \circ P^{\operatorname{tan}}+ (P^{\operatorname{tan}})^*\circ \tilde L_c \circ P^{\operatorname{tan}}.
\end{equation}
This is the content of the main theorem of \cite{BH2015}.

A particular class of metrics inducing this splitting are almost local metrics \cite{Michor2007,Bauer2012a}. Further examples of higher order metrics are given in \cite{BH2015}.

\subsection{The splitting into tangential and constant speed preserving vector fields}
In this section we consider a different splitting, which is motivated by investigations of Riemannian metrics on the space of arc length parametrized curves \cite{Preston2011,Preston2012,PS2013}. In the following lemma we characterize all tangent vectors that preserve constant speed parametrization.
\begin{lem}
Let $c \in \Imm(S^1,\mathbb R^2)$ be parametrized by constant speed. Then a tangent vector $h\in T_c\Imm$ preserves the parametrization of $c$ if and only if 
\begin{equation}\label{equ:volumepreserving}
\langle D_s^2 h,v\rangle + \kappa \langle D_s h,n\rangle=0
\end{equation}
Here  $v=\frac{c'}{|c'|}$ the unit length tangent vector field, $n=iv$ denotes the unit length normal vector field and
 $\kappa$ the curvature of the curve. 
\end{lem}

\begin{proof}
This follows immediately from the infinitesimal action of $h$ on the volume form $ds$:
\begin{equation}
D_{c,h}(ds)= D_{c,h}(|c'|d\theta)=\frac{\langle c', h'\rangle}{|c'|}d\theta= \langle v, D_sh\rangle ds\;.
\end{equation}
Equation~\eqref{equ:volumepreserving} is obtained by setting $D_s\langle v,D_sh\rangle=0$.
\end{proof}
We now describe a decomposition of  $T\Imm(S^1,\R^2)$ in a subspace that preserves constant speed and a complement with values in the tangential bundle.
\begin{lem}
For each regular curve $c$ the  tangent bundle $T_c\Imm(S^1,\R^2)$ can be decomposed as
\begin{equation}\label{TImm_splitting}
T_c\Imm(S^1,\R^2) =   \operatorname{Tan}(c)\oplus \operatorname{Arc}^0(c)\;,
\end{equation}
where 
\begin{align}\label{TImm_splitting2}
 \operatorname{Tan}(c)&:=\left\{h=f.v: a\in C^{\infty}(S^1) \right\}\;,\\
 \operatorname{Arc}^0(c)&:= \left\{h=a.n+b.v: D_s^2 b = D_s (a\kappa) \text{ and } b(0)=0\right\}\;.
\end{align}
The corresponding projections onto these subspaces are given by
\begin{align}\label{TImm_splitting3}
P^{\operatorname{Arc}^0}(k)&:=k^{\operatorname{arc}}=\langle k,n\rangle n + bv\;,\\
P^{\operatorname{Tan}}(k)&:=k^{\operatorname{tan}}= \langle k,v\rangle v- bv\;,
\end{align}
where $b$ solves
\begin{equation}
D_s^2 b = D_s (\langle k,n \rangle \kappa) \text{ with } b(0)=0\;.
\end{equation}
\end{lem}

\begin{proof}
We start by showing that the projections take values in the correct spaces. Let $a=\langle k,n\rangle$. We calculate:
\begin{align*}
\langle D_s P^{\operatorname{Arc^0}}(k), n \rangle&= \langle D_s (an+bv), n \rangle= \langle (D_s a+b \kappa) n+   (D_sb-a\kappa)v, n \rangle
\\&= D_s a+b \kappa\\
\langle D_s^2 P^{\operatorname{Arc^0}}(k), v \rangle&=  \langle D_s^2 (an+bv), v \rangle=\langle D_s ((D_s a+b \kappa) n+   (D_sb-a\kappa)v), v \rangle\\
&=-\kappa (D_s a+b \kappa) +   D^2_s b- D_s(a\kappa) 
\end{align*}
Thus we have
\begin{align*}
\langle D_s^2 h,v \rangle + \kappa \langle D_s h,n\rangle &= -\kappa (D_s a+b \kappa) +   D^2_s b- D_s(a\kappa) +\kappa(D_s a+b \kappa)\\
&=D_s^2 b -D_s (a\kappa)\;.
\end{align*}
This shows that $P^{\operatorname{Arc^0}}(k)$ preserves constant speed parametrization. We choose $b(0)=0$ to enforce uniqueness of the solutions. 
The mapping $P^{\operatorname{Tan}}$ takes values in $\Tan$ by definition. 
The projection property for $P^{\operatorname{Arc^0}}$  is clear. For $P^{\operatorname{Tan}}$ we have
\begin{align*}
&(P^{\operatorname{Tan}})^2(k)=(\operatorname{Id}-P^{\operatorname{Arc}^0})^2(k)=(\operatorname{Id}-P^{\operatorname{Arc}^0})(k)=P^{\operatorname{Tan}}(k)\;.
\end{align*}
Note, that this proves also that  $P^{\operatorname{Tan}}$ (resp. $P^{\operatorname{Arc}^0}$) are surjective 
mappings onto $\operatorname{Tan}$ (resp. $\operatorname{Arc}^0$). 
As $P^{\operatorname{Tan}}$ and $P^{\operatorname{Arc}^0}$ are continuous mappings, their kernels $\operatorname{Arc}$ and $\operatorname{Tan}$ are closed subspaces. This shows that $T\Imm(S^1,\R^2)$ splits as in \eqref{TImm_splitting}.
\end{proof}

Using Theorem \ref{thm:decomposition} we can now construct metrics on $\Imm(S^1,\R^2)$ such that the horizontal bundle coincides with $\operatorname{Arc}^0(c)$. The general formula for these metrics is given by:
\begin{align*}
G_c(h,k) &= \int_{S^1} \langle L_c (h^{\operatorname{tan}}),k^{\operatorname{tan}} \rangle ds+  \int_{S^1} \langle L_c (h^{\operatorname{arc}}),k^{\operatorname{arc}} \rangle ds
\end{align*}
The simplest example, which corresponds to the identity operator $L$, is 
\begin{align*}
G_c(h_1,h_2) &= \int_{S^1} \langle h_1^{\operatorname{tan}},h_2^{\operatorname{tan}} \rangle ds+  \int_{S^1} \langle h_1^{\operatorname{arc}},h_2^{\operatorname{arc}} \rangle ds
\\&= \int_{S^1} \langle h_1-b_1.v,h_2-b_2.v \rangle ds +  \int_{S^1} b_1 b_2+a_1a_2 ds\\
&= \int_{S^1}   
2 a_1 a_2 +\tilde b_1 \tilde b_2 - b_1\tilde b_2 - \tilde b_1 b_2+ b_1 b_2\; ds
\end{align*}
with 
\begin{align*}
a_i = \langle h_i, n\rangle,\quad \tilde b_i = \langle h_i, v\rangle,  \quad D_s^2 b_i = D_s (a_i \kappa), \quad b_i(0)=0\;.
\end{align*}
For $h_1=h_2=h$ this reads as
\begin{align*}
G_c(h,h)= \int_{S^1}   2 a^2  +(\tilde b-b)^2\; ds\;.
\end{align*}
On the space of constant-speed parametrized immersions this induces the $L^2$-metric studied in \cite{Preston2011,Preston2012}.


\end{document}